\documentclass{amsart}

\usepackage{amsfonts}
\usepackage{amssymb}
\usepackage{amsthm}
\usepackage{tikz}

\theoremstyle{plain}
\newtheorem{theorem}{Theorem}[section]
\newtheorem{lemma}[theorem]{Lemma}
\newtheorem{proposition}[theorem]{Proposition}
\newtheorem{corollary}[theorem]{Corollary}

\theoremstyle{definition}
\newtheorem*{definition}{Definition}

\theoremstyle{remark}

\newtheorem*{acknowledgment}{Acknowledgment}

\newcommand{\gen}[1]{\langle#1\rangle}

\newcommand{\Z}{\mathbb{Z}}

\newcommand{\Aut}{\mathrm{Aut}}

\newcommand{\normal}{\trianglelefteq}
\newcommand{\I}{\mathcal{I}}

\renewcommand{\L}{\mathcal{L}}
\newcommand{\ord}{\mathrm{ord}}
\newcommand{\dv}{\big{|}}
\newcommand{\ol}{\overline}
\newcommand{\mat}[4]{\left[\begin{array}{cc}#1&#2\\#3&#4\end{array}\right]}

\begin{document}
\title{Planar infinite groups}
\author{M. Rajabian, M. Farrokhi D. G. and A. Erfanian}

\address{Ferdowsi University of Mashhad, International Campus, Mashhad, Iran.}
\email{mehdi.rajabian@yahoo.com}
\address{Department of Pure Mathematics, Ferdowsi University of Mashhad, Mashhad, Iran.}
\email{m.farrokhi.d.g@gmail.com}
\address{Department of Pure Mathematics and Centre of Excellence in Analysis on Algebraic Structures,
Ferdowsi University of Mashhad, Iran.}
\email{erfanian@math.um.ac.ir}
\date{}

\begin{abstract}
We will determine all infinite $2$-locally finite groups as well as infinite $2$-groups with planar subgroup graph and show that infinite groups satisfying the chain conditions containing an involution do not have planar embeddings. Also, all connected outer-planar groups and outer-planar groups satisfying the chain conditions are presented. As a result, all planar groups which are direct product of connected groups are obtained.
\end{abstract}
\maketitle
\section{Introduction}
Embedding of graphs on surfaces is a central problem in the theory of graphs particularly when the surface has low genus. A lattice can be visualized as a graph in which there is an edge between two elements if one of the elements covers the other. Hence it is natural to ask which lattices can be embedded, as a graph, on a specific surface, say the plane or sphere. A Kuratowski's like theorem for planar lattices is given by Kelly and Rival in \cite{dk-ir} by providing a list of forbidden sublattices. An important family of lattices can be constructed from the subgroups of a group, the \textit{subgroup lattice}, which is studied from various points of view. Drawing of subgroups of a group according to the natural subgroup partial order appears in several contexts, hence it is interesting to know which groups have planar lattices of subgroups. For convenience, the \textit{subgroup graph} of a group $G$ is defined as the graph of its lattice of subgroups, that is, the graph whose vertices are the subgroups of $G$ such that two subgroups $H$ and $K$ are adjacent if one of $H$ or $K$ is maximal in the other. The subgroup graph of $G$ is denoted by $\L(G)$. If $\mathcal{P}$ is a graph theoretical property, then we say that a group has property $\mathcal{P}$ or it is a $\mathcal{P}$-group if the lattice of its subgroups, as a graph, has property $\mathcal{P}$. A group $G$ is planar, or $G$ is a planar group if $\L(G)$ is planar.
We note that subgroups and quotients of planar groups are planar.

Starr and Turner III \cite{cls-get} were the first to study groups $G$ with $\L(G)$ planar and classified all planar abelian groups. Also, Schmidt \cite{rs} and Bohanon and Reid \cite{jpb-lr} simultaneously classified all finite planar groups.

In this paper, we shall study the planarity of infinite non-abelian groups and show that planarity of the lattice of subgroups of a group imposes many restrictions on the structure of the given group. We show that if $G$ is a group such that any two elements of $G$ generate a finite subgroup, then $G$ is planar if and only if it is a finite group or an infinite abelian group as described in Theorems \ref{finitegroups} and \ref{infiniteabeliangroups}, respectively. As a result, it follows that there are no planar infinite solvable groups as well as FC-groups other than planar infinite abelian groups and that all Engel groups satisfying the ascending chain condition are finite. Remind that an FC-group is a group in which all conjugacy classes are finite. Also, we prove that a planar infinite group satisfying the ascending and descending chain conditions, henceforth \textit{chain conditions}, is generated by two elements and it has neither involutions nor elements whose orders are a product of three distinct primes. In contrast to chain conditions, a planar infinite group that is the union of finite pairwise comparable subgroups is abelian (see Theorem \ref{2-locallyfinite}). Therefore, by utilizing the fact that infinite $2$-groups satisfy the normalizer condition on finite subgroups, it can be shown that  every planar infinite $2$-groups is the union of finite pairwise comparable subgroups and hence is ableian. In the remainder of this paper, we shall use the notion of outer-planar graphs. A graph is \textit{outer-planar} if it has a planar embedding in which all vertices lie on the outer (unbounded) region. We classify all outer-planar groups satisfying the chain conditions. Finally, by using the structure of connected outer-planar groups, we show that a planar group that is a product of two nontrivial groups with connected subgroup graphs is isomorphic to one of the groups $\Z_{p^mq}$, $\Z_{p^mqr}$ or $\Z_{p^m}\times\Z_p$ for distinct primes $p,q$ and $r$.

The following theorem of Schmidt, Bohanon and Reid is crucial in our proofs. We note that a group $G$ is called \textit{Hasse planar} if its subgroup graph can be drawn in the $xy$-plane such that whenever $H$ is a maximal subgroup of $K$ for subgroups $H$ and $K$ of $G$, then the $y$-coordinate of $H$ is less than that of $K$.
\begin{theorem}[\cite{jpb-lr,rs}]\label{finitegroups}
Up to isomorphism, the only finite planar groups are the trivial group and
\begin{itemize}
\item[(1)]$\Z_{p^m}$, $\Z_{p^mq^n}$, $\Z_{p^mqr}$, $\Z_{p^m}\times\Z_p$,
\item[(2)]$Q_8=\gen{a,b:a^4=1,b^2=a^2,bab^{-1}=a^{-1}}$,
\item[(3)]$Q_{16}=\gen{a,b:a^8=1, b^2=a^4, bab^{-1}=a^{-1}}$,
\item[(4)]$QD_{16}=\gen{a,b:a^8=b^2=1, bab^{-1}=a^3}$,
\item[(5)]$M_{p^m}=\gen{a,b:a^{p^{m-1}}=b^p=1, bab^{-1}=a^{p^{m-2}+1}}$,
\item[(6)]$\Z_p\rtimes\Z_{q^n}=\gen{a,b:a^p=b^{q^n}=1, bab^{-1}=a^i}$, where $q\dv(p-1)$ and $\ord_p(i)=q$,
\item[(7)]$(\Z_p\times\Z_p)\rtimes\Z_q=\gen{a,b,c:a^p=b^p=c^q=1,ab=ba, cac^{-1}=a^ib^j,cbc^{-1}=a^kb^l}$, where ${\tiny\mat{i}{j}{k}{l}}$ is an element of order $q>2$ in $GL_2(p)$ and $q\dv(p+1)$, 
\end{itemize}
where $p,q$ and $r$ are distinct primes. The groups $\Z_{p^mqr}$ and $QD_{16}$ are the only planar groups that are not Hasse planar.
\end{theorem}

We note that in Theorem \ref{finitegroups}, $M_8\cong D_8$ and that the groups in part (6) include all dihedral groups of order $2p$, where $p$ is a prime. The following theorem of Starr and Turner III will be used in the next section.
\begin{theorem}[\cite{cls-get}]\label{infiniteabeliangroups}
An infinite abelian group is planar if and only if it is isomorphic to one of the groups $\Z_{p^{\infty}}$, $\Z_{p^{\infty}}\times\Z_p$, $\Z_{p^{\infty}}\times\Z_{q^m}$, $\Z_{p^{\infty}}\times\Z_{q^{\infty}}$ or $\Z_{p^{\infty}}\times\Z_q\times\Z_r$, where $p$, $q$ and $r$ are distinct primes and $m$ is a positive integer.  
\end{theorem}
\section{Locally finite groups}
Locally finite groups are the groups which are most closely related to finite groups and are the best candidates for our investigation of planar groups. Indeed we study a more general case. We need the following simple lemma for the proof of Theorem \ref{locallyfinitegroups}.
\begin{lemma}\label{nestedfinitesubgroups}
Let $G$ be a planar group and $H,K$ be non-abelian finite subgroups of $G$ such that $H<K$. Then either $K\cong Q_{16}$ and $H\cong Q_8$, or $K\cong QD_{16}$ and $H\cong D_8$ or $Q_8$.
\end{lemma}
\begin{proof}
The result follows from the fact that the maximal subgroups of groups of types (5), (6) and (7) in Theorem \ref{finitegroups} are abelian and that $Q_{16}$ has subgroups isomorphic to $Q_8$ and $QD_{16}$ has subgroups isomorphic to $D_8$ and $Q_8$.
\end{proof}
\begin{definition}
An \textit{$n$-locally finite group} is a group such that the subgroups generated by every subset with at most $n$ elements is finite. A group that is $n$-locally finite for all $n\geq1$ is called a \textit{locally finite group}.
\end{definition}

It is clear that all subgroups and quotients of $n$-locally finite groups are $n$-locally finite and that all $(n+1)$-locally finite groups (locally finite groups) are $n$-locally finite for all $n\geq1$. In what follows, $\omega(G)$ stands for the set of all orders of elements of a group $G$ and $\pi(G)$ denotes the set of all primes in $\omega(G)$. Also, a Sylow $p$-subgroup of a group $G$ is denoted by $S_p(G)$.
\begin{theorem}\label{2-locallyfinite}
A planar $2$-locally finite group is locally finite.
\end{theorem}
\begin{proof}
Let $\pi^*(G)=\{p\in\pi(G):p^2\in\omega(G)\}$. We claim that $\gen{x}\cap\gen{y}\neq1$ for all $p$-elements $x,y\in G$ of orders $>p$. Suppose on the contrary that $\gen{x}\cap\gen{y}=1$. If $xy=yx$ then $\gen{x,y}\cong\gen{x}\times\gen{y}$, which contradicts Theorem \ref{finitegroups}. Thus $H=\gen{x,y}$ is a non-abelian finite group with $p$-elements $x,y$ of order $>p$ and, by Theorem \ref{finitegroups}, we have the following cases:

(i) $H\cong Q_8$, $Q_{16}$, $QD_{16}$ or $M_{p^k}$ is a $p$-group. 

(ii) $H\cong\Z_q\rtimes\Z_{p^k}$.

If (i) holds, then $H$ has a maximal cyclic subgroup $\gen{z}$, which implies that  $x^p,y^p\in\gen{z}$. Hence $\gen{x^p}\cap\gen{y^p}\neq1$ and consequently $\gen{x}\cap \gen{y}\neq1$, which is a contradiction. Assume that (ii) holds. Then $[S_p(H):C_{S_p(H)}(S_q(H))]=p$, which implies that  $x^p,y^p\in C_{S_p(H)}(S_q(H))$.  Hence $\gen{x^p}\cap\gen{y^p}\neq 1$ so that $\gen{x}\cap\gen{y}\neq 1$, which is a contradiction. Having eliminated (i) and (ii) we have $\gen{x}\cap\gen{y}\neq1$, from which it follows that $\Theta_p(G)=\bigcap_{|x|=p^m>p}\gen{x}$ is a non-trivial normal cyclic subgroup of $G$ for all $p\in\pi^*(G)$. Let $\Theta(G)=\gen{\Theta_p(G):p\in\pi^*(G)}$. Then $\Theta(G)$ is a finite cyclic normal subgroup of $G$, which implies that $|\pi^*(G)|\leq3$. Let $\Theta_0(G)=1$ and define $\Theta_i(G)$ inductively by $\Theta_i(G)/\Theta_{i-1}(G)=\Theta(G/\Theta_{i-1}(G))$ for all $i\geq1$. Since $\Theta_i(G)/\Theta_{i-1}(G)$ is a cyclic group for all $i\geq1$, $T=\bigcup \Theta_i(G)$ is a locally finite group. If $\ol{G}=G/T$ is finite, then $G$ is locally finite and we are done. Thus we may assume that $\ol{G}$ is infinite. 

Since $\Theta(\ol{G})=1$, the elements of $\ol{G}$ have square-free order. If $\ol{x},\ol{y}$ are non-commuting elements of $\ol{G}$, then by Theorem \ref{finitegroups}, $\gen{\ol{x},\ol{y}}\cong\Z_p\rtimes\Z_q$ or $(\Z_p\times\Z_p)\rtimes\Z_q$, which implies that $\ol{x},\ol{y}$ have prime orders. Thus every element with composite order is central. By Theorem \ref{finitegroups}, $Z_\infty(\ol{G})$ is finite and $\ol{G}/Z_\infty(\ol{G})$ is a centerless group. Thus we may assume without loss of generality that $\ol{G}$ is a centerless group with nontrivial elements of prime orders. If the number of $p$-elements of $\ol{G}$ is finite for some $p\in\pi(\ol{G})$, then $\ol{G}$ has a finite Sylow $p$-subgroup $\ol{S}$ with finitely many conjugates. Since $N_{\ol{G}}(\ol{S})/C_{\ol{G}}(\ol{S})\leq\Aut(\ol{S})$ and $C_{\ol{G}}(\ol{S})\leq\ol{S}$, it follows that $N_{\ol{G}}(\ol{S})$ is finite. On the other hand, $[\ol{G}:N_{\ol{G}}(\ol{S})]$ is finite, which implies that $\ol{G}$ is finite, a contradiction. Hence $\ol{G}$ has infinitely many $p$-elements for all primes $p\in\pi(\ol{G})$. Clearly, $\ol{G}$ is not a $p$-group. Let $p=\min\pi(\ol{G})$ and $q\in\pi(\ol{G})\setminus\{p\}$, $\ol{x}$ be a $p$-element and $\ol{Y}$ be the set of all $q$-elements of $\ol{G}$. Then $\gen{\ol{x},\ol{y}}=\gen{\ol{y}}\rtimes\gen{\ol{x}}$ or $\gen{\ol{x},\ol{x}^{\ol{y}}}\rtimes\gen{\ol{y}}$ for all $\ol{y}\in\ol{Y}$. If $\gen{\ol{x},\ol{y}}=\gen{\ol{y}}\rtimes\gen{\ol{x}}$ and $\gen{\ol{x},\ol{z}}=\gen{\ol{z}}\rtimes\gen{\ol{x}}$ for some $\ol{y},\ol{z}\in\ol{Y}$ such that $\gen{\ol{y}}\neq\gen{\ol{z}}$, then $\ol{x}\notin\gen{\ol{y},\ol{z}}$ for otherwise $\ol{y}\ol{z}\neq\ol{z}\ol{y}$ and $\gen{\ol{y},\ol{z}}=\gen{\ol{x},\ol{x}^{\ol{y}}}\rtimes\gen{\ol{y}}=\gen{\ol{x},\ol{y}}=\gen{\ol{y}}\rtimes\gen{\ol{x}}$, which is a contradiction. Therefore $\gen{\ol{x},\ol{y},\ol{z}}$ is a non-abelian finite subgroup of $\ol{G}$ properly containing $\gen{\ol{y},\ol{z}}$, which is impossible by Theorem \ref{finitegroups}. Thus $\gen{\ol{x},\ol{y}}=\gen{\ol{x},\ol{x}^{\ol{y}}}\rtimes\gen{\ol{y}}$ for infinitely many elements $\ol{y}\in\ol{Y}$, from which it follows that $C_{\ol{G}}(\ol{x})$ is infinite. On the other hand, $C_{\ol{G}}(\ol{x})$ is a $p$-group so that $C_{\ol{G}}(\ol{x})$ is an infinite elementary abelian $p$-group leading us to a contradiction. The proof is complete.
\end{proof}
\begin{theorem}\label{locallyfinitegroups}
A planar infinite locally finite group is abelian.
\end{theorem}
\begin{proof}
If $G$ is not abelian, then it has two non-abelian finite subgroups $H,K$ such that $H<K$ and $|H|>16$ contradicting Lemma \ref{nestedfinitesubgroups}. Therefore $G$ is abelian.
\end{proof}
\begin{corollary}\label{solublegroups}
A planar infinite soluble group is abelian.
\end{corollary}
\begin{corollary}\label{FCgroups}
A planar infinite FC-group is abelian.
\end{corollary}
\begin{proof}
The result follows from \cite[Theorem 14.5.10]{djsr}.
\end{proof}
\begin{corollary}\label{Engelgroups}
A planar Engel group satisfying the ascending chain condition is finite.
\end{corollary}
\begin{proof}
The result follows from \cite[Theorem 12.3.7]{djsr}.
\end{proof}

We conclude this section by a partial result on left Engel elements.
\begin{proposition}\label{engelelements}
Let $x$ be a left Engel element of a planar group $G$ which satisfies the ascending chain condition. Then $x$ has finitely many conjugates. In particular, one of the following conditions hold:
\begin{itemize}
\item[(1)]$\gen{x}\normal G$ and $|x^G|\leq\varphi(|x|)$,
\item[(2)]$\gen{x}\not\normal G$, $|x|=p^k\neq4$ ($k>1$) and $|x^G|\leq p^k(p-1)$,
\item[(3)]$\gen{x}\not\normal G$, $|x|=p\neq2$ and $|x^G|\leq p^2-1$,
\item[(4)]$\gen{x}\not\normal G$, $|x|=2$ and $|x^G|\leq5$, or
\item[(5)]$\gen{x}\not\normal G$, $|x|=4$ and $|x^G|\leq6$,
\end{itemize}
where $\varphi$ denotes the Euler totient function.
\end{proposition}
\begin{proof}
First suppose that $x^G$ is infinite. By \cite[Theorem 12.3.7]{djsr}, $x\in F(G)$, the Fitting subgroup of $G$, and hence $F(G)$ is an infinite nilpotent subgroup of $G$. Since $F(G)$ is a locally finite group, by Theorem \ref{locallyfinitegroups}, $F(G)$ is abelian, which contradicts Theorem \ref{infiniteabeliangroups}. Now, suppose that $x^G$ is finite. Since $F(G)$ is nilpotent, $\gen{x^G}$ is nilpotent and by Theorem \ref{finitegroups}, it is isomorphic to $\Z_{p^m}$, $\Z_{p^mq^n}$, $\Z_{p^mqr}$, $\Z_{p^m}\times\Z_p$, $Q_8$, $Q_{16}$, $QD_{16}$ or $M_{p^n}$. A simple verification shows that the size of $x^G$ is bounded above by the numbers given in the proposition and we are done.
\end{proof}
\section{Groups satisfying chain conditions}
In this section, we shall study planar infinite groups satisfying the chain conditions. We note that, by Theorem \ref{infiniteabeliangroups}, such groups are non-abelian. It is worth noting that the class of groups satisfying the chain conditions contains some important infinite simple groups, for instance \textit{Tarski monsters}. A \textit{Tarski group} is a group in which every nontrivial proper subgroup has a fixed prime order. It is well-known that Tarski groups exist for all primes $>10^{75}$ (see \cite{ayo}). Clearly, the subgroup graph of a Tarski group is an infinite bipartite graph $K_{2,\infty}$, where the infinite part is countable. Tarski groups are planar groups satisfying the chain conditions, hence a complete classification of all planar groups satisfying the chain conditions might be very difficult in general. Nevertheless, we are able to give some general properties of such groups. Our first result gives a criterion for the order of elements of such groups whose proof uses the following famous theorem of Platt.
\begin{theorem}[Platt \cite{crp}]
A lattice $L$ is Hasse planar if and only if the graph obtained from joining the minimum and maximum elements is planar.
\end{theorem}
\begin{theorem}\label{pqr}
Let $G$ be a planar infinite group satisfying the chain conditions. Then $G$ has no elements of order $pqr$, for distinct primes $p$, $q$ and $r$.
\end{theorem}
\begin{proof}
Suppose on the contrary that $G$ has an element $x$ of order $pqr$ for some distinct primes $p$, $q$ and $r$. If there is a path in $\L(G)$ from $1$ to $\gen{x}$ disjoint from $\L(\gen{x})$, then by Platt's theorem and Theorem \ref{finitegroups}, $\L(G)$ is not planar contradicting the hypothesis. Thus $\gen{x}\cap\gen{y}\neq1$ for all $y\in G$, otherwise refinements of $1\subset\gen{y}\subset\gen{x,y}$ and $\gen{x}\subset\gen{x,y}$ give rise to a path from $1$ to $\gen{x}$ disjoint from $\L(\gen{x})$. Now, if $y$ is a $p$-element, a $q$-element, or a $r$-element, then $\gen{x^{qr}}\leq\gen{y}$, $\gen{x^{pr}}\leq\gen{y}$, or $\gen{x^{pq}}\leq\gen{y}$, respectively, from which it follows that $\gen{x^{qr}}$, $\gen{x^{pr}}$ and $\gen{x^{pq}}$ are normal subgroups of $G$. Hence $\gen{x}\normal G$ and $G/C_G(x)$ is isomorphic to a subgroup of $\Aut(\gen{x})$. Since $\Aut(\gen{x})$ is finite, $C_G(x)$ is infinite. Let $H=C_G(x)$. Then $x\in Z(H)$ and $\pi(H)=\{p,q,r\}$. For every element $g\in H$ and $s\in\pi(H)$, let $g=g_sg_{s'}$, where, $g_s$ is a $s$-element, $g_{s'}$ is a $s'$-element and $g_sg_{s'}=g_{s'}g_s$. Let $x_0=x$, $h_0=x_s$ and $H_1=H/\gen{h_0}$. If $s\in\pi(H_1)$, then there exists $h_1\gen{h_0}\in H_1$ such that $|h_1\gen{h_0}|=s$. Let $x_1=h_1x_{s'}$. Since $|x_1\gen{h_0}|=pqr$, by the same argument as above $\gen{x_1\gen{h_0}}\normal H_1$ which implies that $\gen{h_1\gen{h_0}}\normal H_1$. Let $H_2=H/\gen{h_0,h_1}$. Continuing this way, we obtain a sequence of groups $\{H_k\}$ such that $H_{k+1}=H_k/\gen{h_0,h_1,...,h_k}$ and $|h_k\gen{h_0,\ldots,h_{k-1}}|=s$ for each $k$. Since $G$ satisfies the chain conditions $\{H_k\}$ is finite, which implies that $s\notin\pi(H_{k+1})$ for some $k$. Hence $S_s=\gen{h_0,h_1,\ldots,h_k}$ is a normal Sylow $s$-subgroup of $H$ for all $s\in\pi(H)$. Clearly, $S_p$, $S_q$ and $S_r$ are finite and hence $H=S_pS_qS_r$ is finite, which is a contradiction.
\end{proof}

In what follows, we will provide all the tools we need to prove that a planar infinite group satisfying the chain conditions has no involutions. In the following lemmas and corollaries, $G$ stands for a planar group satisfying the chain conditions. Also, the set of all involutions of a group $G$ is denoted by $\I(G)$. The following result will be used in the sequel.
\begin{lemma}\label{K3,3 Impossibility}
If $H=\gen{x,y,z}$ is a subgroup of $G$ properly containing $\gen{x,y}$, $\gen{y,z}$, $\gen{z,x}$ and $\gen{x,yz}$, then $\gen{x,yz}=\gen{xyz}$. In particular, $H=\gen{y,xyz}$ is $2$-generated.
\end{lemma}
\begin{proof}
Suppose on the contrary that $\gen{x,yz}\neq\gen{xyz}$. Then $\L(G)$ has a subdivision of $K_{3,3}$ drawn in Figure 1, which is a contradiction. Note that dashed lines indicate paths (maximal chains of subgroups) between the corresponding vertices.
\end{proof}
\begin{center}
\begin{tikzpicture}[scale=1.25]
\node [circle,fill=white,inner sep=1pt] (D) at ({2*cos(30)+2.3},{2*sin(30)}) {};

\node [circle,fill=black,inner sep=1pt,label=below:\tiny{$\gen{y,z}$}] (A) at ({2*cos(270)},{2*sin(270)}) {};

\node [circle,fill=black,inner sep=1pt,label=315:\tiny{$\gen{x,y}\cap\gen{y,z}$}] (B) at ({sqrt(3)*cos(300)},{sqrt(3)*sin(300)}) {};

\node [circle,fill=black,inner sep=1pt,label=right:\tiny{$\gen{x,y}$}] (C) at ({2*cos(330)},{2*sin(330)}) {};

\node [circle,fill=black,inner sep=1pt,label=right:\tiny{$\gen{x,y,z}$}] (D) at ({2*cos(30)},{2*sin(30)}) {};

\node [circle,fill=black,inner sep=1pt,label=above:\tiny{$\gen{x,z}$}] (E) at ({2*cos(90)},{2*sin(90)}) {};

\node [circle,fill=black,inner sep=1pt,label=left:\tiny{$\gen{x,y}\cap\gen{x,z}$}] (F) at ({2*cos(150)},{2*sin(150)}) {};

\node [circle,fill=black,inner sep=1pt,label=left:\tiny{$\gen{x,y}\cap\gen{x,z}\cap\gen{xyz}$}] (G) at ({sqrt(3)/cos(10)*cos(170)},{sqrt(3)/cos(10)*sin(170)}) {};

\node [circle,fill=black,inner sep=1pt,label=left:\tiny{$\gen{xyz}$}] (H) at ({sqrt(3)/cos(10)*cos(190)},{sqrt(3)/cos(10)*sin(190)}) {};

\node [circle,fill=black,inner sep=1pt,label=left:\tiny{$\gen{x,yz}$}] (I) at ({2*cos(210)},{2*sin(210)}) {};

\node [circle,fill=black,inner sep=1pt,label=225:\tiny{$\gen{y,z}\cap\gen{x,yz}$}] (J) at ({sqrt(3)*cos(240)},{sqrt(3)*sin(240)}) {};

\node [circle,fill=black,inner sep=1pt,label=right:\tiny{$\gen{x,z}\cap\gen{y,z}$}] (K) at ({cos(270)},{sin(270)}) {};

\draw [dashed] (A)--(B)--(C)--(D)--(E)--(F)--(G)--(H)--(I)--(J)--(A)--(E);
\draw [dashed] (C)--(F);
\draw [dashed] (D)--(I);
\end{tikzpicture}\\
Figure 1
\end{center}

The above lemma has the following interesting consequences.
\begin{corollary}\label{generators}
The group $G$ is generated by two elements.
\end{corollary}
\begin{proof}
Since $G$ is finitely generated, the result follows by Lemma \ref{K3,3 Impossibility}.
\end{proof}
\begin{corollary}\label{Frattini}
If $x\in G$ such that $\gen{x,y}\subset G$ for all $y\in G$, then $x\in\Phi(G)$.
\end{corollary}
\begin{proof}
Suppose on the contrary that $x\notin\Phi(G)$. Then $G=\gen{x,X}\neq\gen{X}$ for some subset $X$ of $G$. By Lemma \ref{K3,3 Impossibility}, $\gen{X}=\gen{y,z}$ for some $y,z\in G$. Hence $G=\gen{x,y,z}$ and by Lemma \ref{K3,3 Impossibility}, we obtain $\gen{x,yz}=\gen{xyz}$ is abelian so that $[x,yz]=1$. Since $\gen{X}=\gen{y^{-1},yz}$ a same argument shows that $[x,z]=[x,y^{-1}\cdot yz]=1$, which implies that $x\in Z(G)$. Now let $N=\gen{x}\cap\gen{X}$. Then $G/N\cong\gen{x}/N\times\gen{X}/N$ and by Corollary \ref{Product} of the last section, $\gen{X}/N=\gen{gN}$ for some $g\in G$. Hence, $G=\gen{x,g}$ contradicting the assumption on $x$.
\end{proof}
\begin{corollary}\label{Hypercenter}
We have
\begin{itemize}
\item[(1)]$Z(G)\leq\Phi(G)$ if $G$ is non-abelian, and
\item[(2)]$Z_\infty(G)\leq\Phi(G)$ if $G$ is infinite.
\end{itemize}
\end{corollary}
\begin{proof}
(1) It is clear by Corollary \ref{Frattini}. 

(2) We proceed by induction on $n$ to show that $Z_n(G)\subseteq\Phi(G)$ for all $n\geq1$. Since $G$ is non-abelian, by Corollary \ref{Frattini}, $Z_1(G)\subseteq\Phi(G)$. Now, suppose that $Z_n(G)\subseteq\Phi(G)$. Then 
\[\frac{Z_{n+1}(G)}{Z_n(G)}=Z\left(\frac{G}{Z_n(G)}\right)\subseteq\Phi\left(\frac{G}{Z_n(G)}\right)=\frac{\Phi(G)}{Z_n(G)},\]
which implies that $Z_{n+1}(G)\subseteq\Phi(G)$. Therefore $Z_{\infty}(G)=\bigcup_{n=1}^{\infty}Z_n(G)\leq\Phi(G)$, as required.
\end{proof}
\begin{corollary}
If $H$ is a non-abelian subgroup of $G$, then $C_G(H)=Z(H)$.
\end{corollary}
\begin{proof}
Suppose on the contrary that $C_G(H)\neq Z(H)$. Then there exists $z\in C_G(H)$ such that $z\notin H$. Let $K=\gen{H,z}$. Then $z\in Z(K)$ but $z\notin\Phi(K)$ contradicting Corollary \ref{Hypercenter}(1).
\end{proof}
\begin{corollary}\label{center}
$Z(G)=\bigcap Z(H)$, where the intersection is taken over all non-abelian subgroups $H$ of $G$.
\end{corollary}

In the sequel, we assume that the groups under consideration have some involutions and try to reach to a contradiction.
\begin{lemma}\label{Kleinnormality}
If $G$ has a Klein $4$-subgroup $K$, then $C_K(x)\cong\Z_2$ and $K^x=K$ for all $x\in\I(G)\setminus K$.
\end{lemma}
\begin{proof}
Let $K=\gen{a,b}$, $x\in\I(G)\setminus K$ and $H=\gen{K,x}$. By Lemma \ref{K3,3 Impossibility}, $H=\gen{k,x}$ for some $k\in K\setminus\{1\}$ for otherwise $\gen{xab}=\gen{x,ab}$ is a dihedral group, which is impossible. Hence, by Theorem \ref{finitegroups}, $H\cong D_8$. Therefore, $K^x=K$ and $K\cap Z(H)\neq1$, from which the result follows.
\end{proof}

In the following two lemmas, $G$ is assumed to be a planar infinite group satisfying chain conditions.
\begin{lemma}\label{Klein}
$G$ has no Klein $4$-subgroups.
\end{lemma}
\begin{proof}
Let $J=\gen{\I(G)}$ and $K$ be a Klein $4$-subgroup of $G$. If $J$ is finite, then $C_G(J)$ is infinite for $J\normal G$ and $N_G(J)/C_G(J)$ is isomorphic to a subgroup of $\Aut(J)$. By \cite[Theorem 14.4.2]{djsr}, $C_G(J)$ is not a $2$-group and hence $C_G(J)$ has a nontrivial element $x$ of odd order. Then, by Theorem \ref{finitegroups}, $\gen{K,x}\cong K\times\gen{x}$ is not planar, which is a contradiction. Therefore, $J$ is infinite. By Lemma \ref{Kleinnormality}, $K\normal J$. So $J/C_J(K)$ is isomorphic to a subgroup of $\Aut(K)\cong S_3$.  Thus $C_J(K)$ is infinite and, by using the same argument as above, we reach to a contradiction.
\end{proof}
\begin{lemma}\label{infinitecentralizer}
If $G$ has an involution, then it has an involution with infinite centralizer.
\end{lemma}
\begin{proof}
If $\I(G)$ is finite then we are done. Thus we may assume that $\I(G)$ is infinite. Let $a,b\in\I(G)$ be two distinct involutions. We have two cases:

(i) $b^c\in\gen{a,b}$ for all $c\in\I(G)\setminus\I(\gen{a,b})$. Since $\I(\gen{a,b})$ is finite, there exist distinct elements $c,c_1,c_2,\ldots\in\I(G)\setminus\I(\gen{a,b})$ such that $b^c=b^{c_1}=b^{c_2}=\cdots$. Thus $cc_i\in C_G(b)$ for all $i$, which implies that $C_G(b)$ is infinite.

(ii) There exists $c\in\I(G)\setminus\I(\gen{a,b})$ such that $b^c\notin\gen{a,b}$. Let $H=\gen{a,b,c}$. If $H$ is a dihedral group, then by Theorem \ref{finitegroups} and Lemma \ref{Klein}, $H\cong D_{2p}$ for some prime $p$. But then $\gen{a,b}=H$ and $c\in\gen{a,b}$, which is a contradiction. Now, it is easy to see that $\gen{a,b}$, $\gen{a,c}$, $\gen{b,c}$, $\gen{a,b^c}$ are distinct and maximal chains from $\gen{a,c}$, $\gen{a,b}$ and $\gen{a,b^c}$ to $H$ are disjoint. Therefore, the subgraph induced by $1$, $\gen{a}$, $\gen{b}$, $\gen{c}$, $\gen{b^c}$, $\gen{a,b}$,$\gen{a,c}$, $\gen{b,c}$, $\gen{a,b^c}$ and $\L(H)$ contains a subdividion of $K_{3,3}$, which is a contradiction (see Figure 2).
\end{proof}
\begin{center}
\begin{tikzpicture}[scale=1]
\node [circle,fill=black,inner sep=1pt] (A) at (2,0) {};
\node at (2+0.4,0) {\tiny{$\gen{a,b}$}};

\node [circle,fill=black,inner sep=1pt] (B) at ({2*cos(60)},{2*sin(60)}) {};
\node at ({(2+0.25)*cos(60)},{(2+0.25)*sin(60)}) {\tiny{$\gen{b}$}};

\node [circle,fill=black,inner sep=1pt] (C) at ({2*cos(120)},{2*sin(120)}) {};
\node at ({(2+0.25)*cos(120)},{(2+0.25)*sin(120)}) {\tiny{$1$}};

\node [circle,fill=black,inner sep=1pt] (D) at ({2*cos(180)},{2*sin(180)}) {};
\node at ({(2+0.25)*cos(180)},{(2+0.25)*sin(180)}) {\tiny{$\gen{a}$}};

\node [circle,fill=black,inner sep=1pt] (E) at ({2*cos(240)},{2*sin(240)}) {};
\node at ({(2+0.25)*cos(240)},{(2+0.25)*sin(240)}) {\tiny{$\gen{a,c}$}};

\node [circle,fill=black,inner sep=1pt] (F) at ({2*cos(300)},{2*sin(300)}) {};
\node at ({(2+0.25)*cos(300)},{(2+0.25)*sin(300)}) {\tiny{$H$}};

\node [circle,fill=black,inner sep=1pt] (G) at ({1*cos(60)},{1*sin(60)}) {};
\node at ({(1+0.25)*cos(45)},{(1+0.25)*sin(45)}) {\tiny{$\gen{b,c}$}};

\node [circle,fill=black,inner sep=1pt] (H) at ({1*cos(120)},{1*sin(120)}) {};
\node at ({(1+0.25)*cos(135)},{(1+0.25)*sin(135)}) {\tiny{$\gen{b^c}$}};

\node [circle,fill=black,inner sep=1pt] (I) at ({1*cos(240)},{1*sin(240)}) {};
\node at ({(1+0.15)*cos(228)},{(1+0.15)*sin(228)}) {\tiny{$\gen{c}$}};

\node [circle,fill=black,inner sep=1pt] (J) at ({1*cos(300)},{1*sin(300)}) {};
\node at ({(1+0.3)*cos(318)},{(1+0.3)*sin(318)}) {\tiny{$\gen{a,b^c}$}};

\draw (C)--(H)--(J);
\draw (B)--(G)--(I)--(E)--(D)--(C)--(B)--(A)--(D);
\draw [dashed] (A)--(F)--(E);
\draw [dashed] (F)--(J);
\end{tikzpicture}\\
Figure 2
\end{center}
\begin{theorem}
A planar infinite group satisfying the chain conditions has no involutions.
\end{theorem}
\begin{proof}
Suppose on the contrary that $G$ has an involution. By Lemma \ref{infinitecentralizer}, $G$ has an involution $x$ such that $C_G(x)$ is infinite. Since $G$ satisfies the ascending chain condition, by Lemma \ref{infinitecentralizer}, we may assume that $C_G(x)/\gen{x}$ has no involutions for otherwise we may construct an infinite series of subgroups $\gen{x_1}\subset\gen{x_1,x_2}\subset\cdots$ such that $x_i\gen{x_1,\ldots,x_{i-1}}$ is involution and 
\[\frac{G_i}{\gen{x_1,\ldots,x_{i-1}}}=C_{\frac{G_{i-1}}{\gen{x_1,\ldots,x_{i-1}}}}(x_i\gen{x_1,\ldots,x_{i-1}})\]
 is infinite for all $i\geq1$, where $G_0=G$.

If $H/\gen{x}$ is a finite subgroup of $C_G(x)/\gen{x}$, then $|H/\gen{x}|$ is odd and there exists a subgroup $K$ of $C_G(x)$ such that $H=\gen{x}\times K$. By Theorems \ref{finitegroups} and \ref{pqr}, $K\cong\Z_{p^m}$ is a cyclic $p$-group. Hence every finite subgroup of $C_G(x)/\gen{x}$ is a cyclic $p$-group for some prime $p$. If $y,z\in C_G(x)$ are elements of odd orders such that $\gen{y}\cap\gen{z}=\gen{z}\cap\gen{yz}=\gen{yz}\cap\gen{y}=1$, then as it is drawn in Figure 3, $\L(C_G(x))$ has a subdivision of $K_{3,3}$, which is a contradiction. As before, dashed lines indicate paths (maximal chains of subgroups) between the corresponding vertices.

Let $C_1=C_G(x)$. If there exist two distinct elements $y,z$ of odd prime orders such that $\gen{y}\neq\gen{z}$, then $\gen{y}\cap\gen{z}=\gen{z}\cap\gen{yz}=\gen{yz}\cap\gen{y}=1$. Therefore, by the same reason as above, $C_1$ has a subgraph isomorphic to a subdivision of $K_{3,3}$, which is a contradiction. Hence $C_1/\gen{x}$ is a $p$-group and $C_1$ contains a unique cycle $\gen{x_1}$ of order $p$. Let $C_2=C_1/\gen{x_1}$. The same argument shows that there exists a unique cycle $\gen{x_2\gen{x_1}}$ in $C_2$ such that $|x_2\gen{x_1}|=p$. By a repetitive argument, there exist elements $x_1,x_2,\ldots,x_n,\ldots$ such that $|x_i\gen{x_1,x_2,x_3,\ldots,x_{i-1}}|=p$ for all $i\geq1$. Therefore $G$ has an infinite ascending chain of subgroups
\[\gen{x_1}\subset\gen{x_1,x_2}\subset\cdots\subset\gen{x_1,\ldots,x_n}\subset\cdots,\]
which contradicts the assumption. The proof is complete.
\end{proof}
\begin{center}
\begin{tikzpicture}[scale=1]
\node [circle,fill=black,inner sep=1pt] (A) at (2,0) {};
\node at (2+0.25,0) {\tiny{$\gen{x}$}};

\node [circle,fill=black,inner sep=1pt] (B) at ({2*cos(60)},{2*sin(60)}) {};
\node at ({(2+0.25)*cos(60)},{(2+0.25)*sin(60)}) {\tiny{$\gen{x,z}$}};

\node [circle,fill=black,inner sep=1pt] (C) at ({sqrt(3)/cos(10)*cos(90)},{sqrt(3)/cos(10)*sin(90)}) {};
\node at ({(sqrt(3)/cos(10)+0.25)*cos(90)},{(sqrt(3)/cos(10)+0.25)*sin(90)}) {\tiny{$\gen{z}$}};

\node [circle,fill=black,inner sep=1pt] (D) at ({2*cos(120)},{2*sin(120)}) {};
\node at ({(2+0.25)*cos(120)},{(2+0.25)*sin(120)}) {\tiny{$1$}};

\node [circle,fill=black,inner sep=1pt] (E) at ({sqrt(3)/cos(10)*cos(150)},{sqrt(3)/cos(10)*sin(150)}) {};
\node at ({(sqrt(3)/cos(10)+0.25)*cos(150)},{(sqrt(3)/cos(10)+0.25)*sin(150)}) {\tiny{$\gen{y}$}};

\node [circle,fill=black,inner sep=1pt] (F) at ({2*cos(180)},{2*sin(180)}) {};
\node at ({(2+0.4)*cos(180)},{(2+0.4)*sin(180)}) {\tiny{$\gen{x,y}$}};

\node [circle,fill=black,inner sep=1pt] (G) at ({2*cos(240)},{2*sin(240)}) {};
\node at ({(2+0.25)*cos(240)},{(2+0.25)*sin(240)}) {\tiny{$\gen{x,y,z}$}};

\node [circle,fill=black,inner sep=1pt] (H) at ({2*cos(300)},{2*sin(300)}) {};
\node at ({(2+0.25)*cos(300)},{(2+0.25)*sin(300)}) {\tiny{$\gen{x,(yz)^2}$}};

\node [circle,fill=black,inner sep=1pt] (I) at ({1*cos(300)},{1*sin(300)}) {};
\node at ({0.85*cos(270)},{0.85*sin(270)}) {\tiny{$\gen{(yz)^2}$}};

\draw (B)--(C);
\draw (E)--(F);
\draw (H)--(I);
\draw [dashed] (C)--(D)--(E);
\draw [dashed] (D)--(I);
\draw [dashed] (H)--(A)--(F)--(G)--(B)--(A);
\draw [dashed] (G)--(H);
\end{tikzpicture}\\
Figure 3
\end{center}

Regardless of chain conditions, it is yet possible to decide on the planarity of a group with involutions whenever it is a $2$-group. Recall that  for a $p$-group $G$, $\Omega_i(G)$ denotes the subgroup of $G$ generated by all elements of orders at most $p^i$ for all $i\geq0$.
\begin{theorem}
A planar infinite $2$-group is abelian.
\end{theorem}
\begin{proof}
First we show that all finite subgroups of $G$ are abelian. Suppose on the contrary that $G$ has a finite non-abelian subgroup $H$. Then, by \cite[Theorem 14.4.1]{djsr}, we may assume that $|H|>16$. Hence, by Theorem \ref{finitegroups}, $H\cong M_{2^m}$ for some $m$. With the same reason, there exists a finite non-abelian subgroup $K\cong M_{2^n}$ properly containing $H$, which contradicts Lemma \ref{nestedfinitesubgroups}. Now, if $x,y\in G$ are two involutions, then $\gen{x,y}$ is a finite dihedral group, which implies that $\gen{x,y}$ is abelian, that is, $xy=yx$. Thus $\Omega_1(G)$ is an elementary abelian $2$-group, from which by Theorem \ref{finitegroups}, it follows that $\Omega_1(G)\cong\Z_2$ or $\Z_2\times\Z_2$. Now, let $A_0=1$ and $A_i$ be a subgroup of $G$ defined inductively by $A_i/A_{i-1}=\Omega_1(G/A_{i-1})$ for all $i\geq1$. Then $A_i$'s are finite abelian subgroups of $G$ such that $G=\bigcup A_i$. Therefore $G$ is abelian, as required.
\end{proof}

Albeit the classification of all planar groups satisfying the chain conditions seems to be a difficult task, the situation is much more convenient for a stronger notion of planarity known as outer-planarity. By a result of Chartrand and Harary \cite{gc-fh}, a finite connected graph is outer-planar if and only if it has no subdivisions of $K_4$ and $K_{2,3}$ as subgraphs. Therefore, a graph with a subdivision of $K_4$ or $K_{2,3}$ is never an outer-planar graph.
\begin{lemma}\label{finiteouterplanargroups}
A finite group is outer-planar if and only if it is isomorphic to $\Z_{p^m}$ or $\Z_{p^mq}$ for some distinct primes $p,q$.
\end{lemma}
\begin{theorem}
Every outer-planar group satisfying the chain conditions is finite (and hence is isomorphic to $\Z_{p^m}$ or $\Z_{p^mq}$, where $p,q$ are distinct primes).
\end{theorem}
\begin{proof}
Suppose $G$ is a non-abelian group. Since $G$ satisfies the chain conditions, there exist an integer $n$ such that $Z(\ol{G})=1$, where $\ol{G}=G/Z_n(G)$. By Corollary \ref{generators}, $\ol{G}=\gen{\ol{x},\ol{y}}$ for some elements $\ol{x},\ol{y}\in\ol{G}$, which implies that $\ol{G}=\gen{\ol{x},\ol{x}\ol{y}}=\gen{\ol{y},\ol{x}\ol{y}}$. It is easy to see that $\gen{\ol{x}}\cap\gen{\ol{y}}=\gen{\ol{x}}\cap\gen{\ol{x}\ol{y}}=\gen{\ol{y}}\cap\gen{\ol{x}\ol{y}}=1$. Therefore $\ol{G},\gen{\ol{x}},\gen{\ol{y}},\gen{\ol{x}\ol{y}},1$ along with subgroups in maximal chains connecting $1$ and $G$ to $\gen{\ol{x}}$, $\gen{\ol{y}}$ and $\gen{\ol{x}\ol{y}}$ give rise to a subdivision of $K_{2,3}$, which is a contradiction. Therefore, $G$ is abelian and by Theorem \ref{infiniteabeliangroups}, $G$ is a finite group. Now, the result follows by Theorem \ref{finiteouterplanargroups}.
\end{proof}
\section{Product of groups}
A theorem of Behzad and Mahmoodian \cite{mb-sem} states that the Cartesian product of two finite connected graphs is planar if they are both paths, one of them is a path and the other is a cycle, or one of them is an outer-planar graph and the other is a single edge. Hence the classification of such graphs relies on the classification of outer-planar graphs. To end this, we first determine all connected outer-planar groups.
\begin{theorem}\label{Outer-planar}
Let $G$ be a connected outer-planar group. Then $G$ is isomorphic to one of the groups $\Z_{p^m}$ or $\Z_{p^mq}$ for some distinct primes $p$ and $q$.
\end{theorem}
\begin{proof}
First we observe that if $G$ is finite, then by Theorem \ref{finitegroups}, $G$ is a cyclic group of order $p^m$ or $p^mq$ for some primes $p$ and $q$.  Thus we may assume that $G$ is infinite. Since $\L(G)$ is connected, $G$ has an infinite subgroup $H$ with a maximal finite cyclic subgroup $\gen{x}$. Clearly, $|x|=p^m$ or $p^mq$. We claim that $\gen{x}$ has infinitely many conjugates in $H$. If not then either $\gen{x}\subset N_H(\gen{x})\subseteq H$ or $N_H(\gen{x})=H$, that is, $\gen{x}\normal H$ and hence $\gen{x}\subset\gen{x,y}\subset H$ for all $y\in H\setminus \gen{x}$, which contradict maximality of $\gen{x}$ in $H$. Now, a simple verification shows that there exist three elements $a,b,c\in H$ such that $\gen{x}^a$, $\gen{x}^b$ and $\gen{x}^c$ are distinct and their pairwise intersections are equal to a fix subgroup $N$. Therefore, $\L(G)$ has a subgraph isomorphic to a subdivision of $K_{2,3}$, which is a contradiction. The proof is complete.
\end{proof}
\begin{corollary}\label{Product}
Let $G=H\times K$ be a planar group, where $H$ and $K$ are non-trivial groups with connected subgroup graph. Then $G$ is isomorphic to one of the groups $\Z_{p^mq}$, $\Z_{p^mqr}$ or $\Z_{p^m}\times\Z_p$ for distinct primes $p,q$ and $r$.
\end{corollary}
\begin{proof}
Since $\L(G)$ has a subgraph isomorphic to $\L(H)\times\L(K)$, the Cartesian product of $\L(H)$ and $\L(K)$, $\L(H)\times\L(K)$ is planar and by \cite[Theorems 4.1 and 4.2]{mb-sem}, it follows that $\L(H)$ and $\L(K)$ are outer-planar. Hence, in all cases, $H,K$ are cyclic groups of orders $p^m$ or $p^mq$ for some primes $p$ and $q$, from which the result follows.
\end{proof}
\begin{corollary}
Let $G$ be an outer-planar infinite group. Then the infinite subgroups of $G$ have no maximal finite subgroups. 
\end{corollary}
\begin{corollary}
If $G$ is an outer-planar infinite group the girth of its subgroup graph is greater than four, then the finite subgroups of $G$ induce a tree as a connected component.
\end{corollary}
\begin{acknowledgment}
The authors would like to thank the referee for careful reading of the paper and giving some helpful comments and pointing out a gap in the proof of Lemma \ref{K3,3 Impossibility}.
\end{acknowledgment}


\begin{thebibliography}{0}
\bibitem{mb-sem}M. Behzad and S. E. Mahmoodian, On topological invariants of the product of graphs, \textit{Canad. Math. Bull.} \textbf{12}(2) (1969), 157--166.

\bibitem{jpb-lr}J. P. Bohanon and L. Reid, Finite groups with planar subgroup lattices, \textit{J. Algebraic Combin.} \textbf{23} (2006), 207--223.

\bibitem{gc-fh}G. Chartrand and F. Harary, Planar permutation graphs, \textit{Les Annales de I' Institut Henri Poincar\'{e}} \textbf{3}(4) (1967), 433--438.

\bibitem{dk-ir}D. Kelly and I. Rival, Planar lattices, \textit{Can. J. Math.} \textbf{27} (1975), 636--665.

\bibitem{ayo}A. Yu. Olshanskii, An infinite group with subgroups of prime orders, \textit{Math. USSR Izv.} \textbf{16} (1981), 279--289.

\bibitem{crp}C. R. Platt, Planar lattices and planar graphs, \textit{J. Combin. Theory Ser. B} \textbf{21} (1976), 30--39.

\bibitem{djsr}D. J. S. Robinson, \textit{A Course in the Theory of Groups}, Second Edition, Spring-Verlag, New York, 1996.

\bibitem{rs}R. Schmidt, Planar subgroup lattices, \textit{Algebra Universalis} \textbf{55} (2006), 3--12.

\bibitem{cls-get}C. L. Starr and G. E. Turner III, Planar groups, \textit{J. Algebraic Combin.} \textbf{19} (2004), 283--295.
\end{thebibliography}
\end{document}